\documentclass[a4paper, 11pt]{article}
\usepackage{amssymb,amsmath,amsthm,stmaryrd,mathrsfs}

\newtheorem{lemma}{Lemma}
\newtheorem{thm}{Theorem}
\newtheorem{cor}{Corollary}
\newtheorem{rmk}{Remark}

\newcommand{\op}[1]{\operatorname{\text{\rm #1}}}

\begin{document}

\title{Regularity of minimal hypersurfaces with a common free boundary}
\author{Brian Krummel}
\maketitle

\begin{abstract} 
Let $N$ be a Riemannian manifold and consider a stationary union of three or more $C^{1,\mu}$ hypersurfaces-with-boundary $M_k \subset N$ with a common boundary $\Gamma$.  We show that if $N$ is smooth, then $\Gamma$ is smooth and each $M_k$ is smooth up to $\Gamma$ (real analytic in the case $N$ is real analytic).  Consequently we strengthen a result of Wickramasekera in~\cite{Wic} to conclude that under the stronger hypothesis that $V$ is a stationary, stable, integral $n$-varifold in an $(n+1)$-dimensional, smooth (real analytic) Riemannian manifold such that the support of $\|V\|$ is nowhere locally the union of three or more smooth (real analytic) hypersurfaces-with-boundary meeting along a common boundary, the singular set of $V$ is empty if $n \leq 6$, is discrete if $n = 7$, and has Hausdorff dimension at most $n-7$ if $n \geq 8$. 
\end{abstract}

\section{Introduction}

We consider the regularity of integral $n$-varifolds of an $(n+1)$-dimensional, smooth, Riemannian manifold $(N,g)$ of the form 
\begin{equation} \label{V_form}
	V = \sum_{k=1}^q \theta_k |M_k| 
\end{equation}
consisting of distinct $C^1$ embedded hypersurfaces-with-boundary $M_1,M_2,\ldots,M_q$ with a common boundary $\Gamma$ and with respective positive multiplicities $\theta_1,\theta_2,\ldots,\theta_q$.  Here $|M_k|$ denotes the multiplicity one varifold associated with the hypersurface $M_k$ and the sum in (\ref{V_form}) is taken by regarding varifolds as Radon measures on the Grassmanian of $N$~\cite[Section 38]{GMT}.  Thus $V$ is the integral varifold supported on $M_1 \cup M_2 \cup \cdots \cup M_q$ with multiplicity $\sum_{k \text{ with } X \in M_k} \theta_k$ at $\mathcal{H}^n$-a.e. $X \in M_1 \cup M_2 \cup \cdots \cup M_q$.  To each integral $n$-varifold $V$ we associate a Radon measure $\|V\|$ such that for any Borel set $A \subseteq N$, $\|V\|(A)$ represents the $n$-dimensional area of $V$ in $A$ (see~\cite[Section 15]{GMT} noting that $\|V\| = \mu_V$).  When $V$ is of the form (\ref{V_form}), 
\begin{equation*}
	\|V\|(A) = \sum_{k=1}^q \theta_k \mathcal{H}^n(M_k \cap A) 
\end{equation*}
for every Borel set $A \subseteq N$, where $\mathcal{H}^n$ denotes the $n$-dimensional Hausdorff measure.  Let $\zeta \in C_c^1(N;TN)$ be an arbitrary vector field on $N$ and let $f_t : N \rightarrow N$, $t \in (-1,1)$, be the one-parameter family of diffeomorphisms on $N$ generated by $\zeta$.  The first variation of area $\delta V : C_c^1(N;TN) \rightarrow \mathbb{R}$ of an integral $n$-varifold $V$ is the linear functional given by 
\begin{equation*}
	\delta V(\zeta) = \left. \frac{d}{dt} \|f_{t \#} V\|(\op{spt} \zeta) \right|_{t=0} = \int_M \op{div}_{T_X} \zeta(X) d\|V\|(X)
\end{equation*}
for every vector field $\zeta \in C_c^1(N;TN)$~\cite[Section 16]{GMT}, where $f_{t \#} V$ denote the image or pushforward of $V$ under the diffeomorphisms $f_t$, $T_X$ is the approximate tangent plane to $V$ at a.e. $X \in \op{spt} \|V\|$, and 
\begin{equation*}
	\op{div}_{T_X} \zeta(X) = \sum_{i=1}^n D_{\tau_i} \zeta(X) \cdot \tau_i
\end{equation*}
for any orthonormal basis $\tau_1,\tau_2,\ldots,\tau_n$ of $T_X$.  Suppose $V$ is of the form (\ref{V_form}) and each $M_k$ is $C^2$ on its interior.  Then 
\begin{equation*}
	f_{t \#} V = \sum_{k=1}^q \theta_k |f_t(M_k)| \hspace{3mm} \text{and} \hspace{3mm} 
	\|f_{t \#} V\|(A) = \sum_{k=1}^q \theta_k \mathcal{H}^n(f_t(M_k) \cap A) \text{ for all Borel sets } A \subseteq N,
\end{equation*}
where $f_t$ are the diffeomorphisms generated by $\zeta$ from above.  Thus by the divergence theorem, 
\begin{equation*}
	\delta V(\zeta) = \sum_{k=1}^q \int_{M_k} \op{div}_{T_X M_k} \zeta(X) \theta_k d\mathcal{H}^n(X) 
	= \int_{\Gamma} \sum_{k=1}^q \theta_k \eta_k \cdot \zeta - \sum_{k=1}^q \int_{M_k} H_k \cdot \zeta \theta_k d\mathcal{H}^n
\end{equation*}
for every vector field $\zeta \in C_c^1(N;TN)$, where $T_X M_k$ denotes the tangent plane to $M_k$ at $X \in M_k$, $H_k$ is the mean curvature vector to $M_k$, and $\eta_k$ is unit vector field along $\Gamma$ that is tangent to $M_k$, orthogonal to $\Gamma$, and points outward from $M_k$.  In particular, $V$ is stationary, i.e. $\delta V(\zeta) = 0$ for all $\zeta \in C_c^1(N;TN)$, if and only if $H_k = 0$ on $M_k$ for all $k = 1,2,\ldots,q$ and 
\begin{equation} \label{mainthm_tangents}
	\sum_{k=1}^q \theta_k \eta_k = 0 \text{ on } \Gamma. 
\end{equation}
Our main result is the following regularity theorem for stationary $n$-varifolds $V$ of the form (\ref{V_form}).

\begin{thm} \label{mainthm}
Let $(N,g)$ be an $(n+1)$-dimensional, smooth (real analytic), Riemannian manifold.  Let $Z \in N$, $\mathcal{O}$ be an open neighborhood of $Z$ in $N$, and $\mu \in (0,1)$.  Let $q \geq 3$ and consider the stationary integral $n$-varifold $V$ of the form (\ref{V_form}) for distinct $C^{1,\mu}$ embedded hypersurfaces-with-boundary $M_1,M_2,\ldots,M_q$ of $\mathcal{O}$ that have a common boundary $\Gamma$ with $Z \in \Gamma$ and for positive integers $\theta_1,\theta_2,\ldots,\theta_q$.  Assume the hypersurfaces $M_k$ are not all tangent to the same hyperplane at $Z$.  Then for some open neighborhood $\mathcal{O}'$ of $Z$ in $\mathcal{O}$, $M_k$ are smooth (real analytic) hypersurfaces-with-boundary of $\mathcal{O}'$ and $\Gamma$ is a smooth (real analytic) $(n-1)$-dimensional submanifold of $\mathcal{O}'$. 
\end{thm}

\begin{rmk} 
Note that Theorem \ref{mainthm} allows for the hypersurfaces $M_k$ to intersect away from $\Gamma$.  If $\overline{M_k} \cap \overline{M_l} \cap \mathcal{O} = \Gamma$ whenever $k \neq l$, then the hypersurfaces $M_k$ cannot all be tangent to the same hyperplane at $Z$ as a standard consequence of the Hopf boundary point lemma~\cite[Lemma 7]{FG}.
\end{rmk}

\begin{rmk} \label{remark2}
Theorem \ref{mainthm} continues to hold if we weaken the assumption that each $M_k$ is a $C^{1,\mu}$ hypersurface-with-boundary to the assumption that each $M_k$ is a $C^1$ hypersurface-with-boundary since then each $M_k$ is automatically a $C^{1,\mu}$ hypersurface-with-boundary for all $\mu \in (0,1/2)$ as a consequence of the proof of the Minimum Distance Theorem of~\cite{Wic} (see the appendix Section \ref{sec:appendix}).  Note that this requires the assumption that the hypersurfaces $M_k$ are not all tangent to the same hyperplane along $\Gamma$. 
\end{rmk}

An important consequence of Theorem \ref{mainthm} is that we can strengthen the Regularity and Compactness Theorem of Wickramasekera in~\cite{Wic}.  Suppose $V$ is a stationary integral $n$-varifold in $N$.  We let $\op{reg} V$ denote the set of points $X_0 \in \op{spt} \|V\|$ such that for some open neighborhood $\mathcal{O}$ of $X_0$ in $N$, $\op{spt} \|V\| \cap \mathcal{O}$ is an smooth embedded hypersurface of $\mathcal{O}$.  We let $\op{sing} V = \op{spt} \|V\| \setminus \op{reg} V$.  We say $\op{reg} V$ is stable if 
\begin{equation*}
	\int_{\op{reg} V} (|A|^2 + \op{Ric}(\nu,\nu)) \varphi^2 \leq \int_{\op{reg} V} |\nabla \varphi|^2 
\end{equation*}
for every $\varphi \in C_c^1(\op{reg} V;\mathbb{R})$, where $|A|$ is the norm of the second fundamental form of $\op{reg} V$, $\op{Ric}$ is the Ricci curvature tensor on $N$, $\nu(X) \in T_X N$ is a unit normal vector to $\op{reg} V$ at each $X \in \op{reg} V$, and $\nabla$ is the gradient on $\op{reg} V$.  The Regularity and Compactness Theorem of~\cite{Wic} states that for $\mu \in (0,1)$, if $V$ is an integral $n$-varifold in an $(n+1)$-dimensional, smooth, Riemannian manifold $(N,g)$ such that $V$ is stationary, $\op{reg} V$ is stable, and there is no open set $\mathcal{O}$ in $N$ for which $\op{spt} \|V\| \cap \mathcal{O}$ is the union of a finite number of three or more $C^{1,\mu}$ embedded hypersurfaces-with-boundary of $\mathcal{O}$ that have a common boundary $\Gamma$ and that do not intersect except along $\Gamma$, then $\op{sing} V = \emptyset$ if $n \leq 6$, $\op{sing} V$ is discrete if $n = 7$, and $\dim \op{sing} V \leq n-7$ if $n \geq 8$.

\begin{cor}
Let $(N,g)$ be an $(n+1)$-dimensional, smooth (real analytic), Riemannian manifold.  Suppose $V$ be an integral $n$-varifold in $N$ such that $V$ is stationary, $\op{reg} V$ is stable, and there is no open set $\mathcal{O}$ in $N$ such that $\op{spt} \|V\| \cap \mathcal{O}$ is the union of a finite number of three or more smooth (real analytic) embedded hypersurfaces-with-boundary of $\mathcal{O}$ that have a common boundary $\Gamma$ and that do not intersect except along $\Gamma$.  Then $\op{sing} V = \emptyset$ if $n \leq 6$, $\op{sing} V$ is discrete if $n = 7$, and $\dim \op{sing} V \leq n-7$ if $n \geq 8$. 
\end{cor}

Theorem \ref{mainthm} can be regarded as analogous to Theorem 5.1 of Kinderlehrer, Nirenberg, and Spruck in~\cite{KNS}, which shows that if $M_1$, $M_2$, and $M_3$ are $C^{1,\mu}$ embedded minimal hypersurfaces with common boundary $\Gamma$ meeting at constant nonzero angles along $\Gamma$, then $M_1$, $M_2$, and $M_3$ are real analytic up to the boundary $\Gamma$ and $\Gamma$ is a real analytic $(n-1)$-dimensional submanifold.  This includes the special case of Theorem \ref{mainthm} where $q = 3$ since if $V$ is stationary then $M_1$, $M_2$, and $M_3$ meet at constant angles along $\Gamma$.  Our approach allows us to prove Theorem \ref{mainthm} for all $q \geq 3$.

The proof of Theorem \ref{mainthm} involves representing the hypersurfaces $M_k$ as the graphs of functions $u_k$ that satisfy a particular free boundary problem.  We replace the boundary condition of~\cite{KNS} that $M_1$, $M_2$, and $M_3$ meet at constant angles along $\Gamma$ with the condition that (\ref{mainthm_tangents}) holds true.  We then transform the free boundary problem using the partial Legendre transform as in~\cite{KNS} and then apply the elliptic regularity theory of Morrey (see~\cite{Morrey}).  The main challenge is verifying that the transformed differential system is coercive (i.e. satisfies the complementing condition of Morrey in~\cite{Morrey}).

Much like with Theorem 5.1 of~\cite{KNS}, Theorem \ref{mainthm} admits a number of extensions which follow from the same proof with slight modifications.  For example, we have the following result analogous to Theorem 5.2 of~\cite{KNS}:

\begin{thm} \label{thm2}
Let $(N,g)$ be an $(n+1)$-dimensional, smooth (real analytic), Riemannian manifold.  Let $Z \in N$, $\mathcal{O}$ be an open neighborhood of $Z$ in $N$, and $\mu \in (0,1)$.  Let $q \geq 3$ and $M_1,M_2,\ldots,M_q$ be distinct $C^{1,\mu}$ embedded hypersurfaces-with-boundary of $\mathcal{O}$ that have a common boundary $\Gamma$ with $Z \in \Gamma$.  Suppose each $M_k$ is $C^2$ on its interior and the mean curvature vector of $M_k$ equals $\Lambda_k(X) \nu_k(X)$ at every $X \in M_k$ for some smooth (real analytic) function $\Lambda_k : \mathcal{O} \rightarrow \mathbb{R}$ and continuous unit normal vector field $\nu_k$ on $M_k$.  Further suppose 
\begin{equation*}
	\sum_{k=1}^q \theta_k \eta_k = 0 \text{ on } \Gamma 
\end{equation*}
for some positive, smooth (real analytic) functions $\theta_k : \mathcal{O} \rightarrow \mathbb{R}_+$, where $\eta_k$ is the unit vector field on $\Gamma$ that is tangent to $M_k$, orthogonal to $\Gamma$, and points outward from $M_k$.  Assume the surfaces $M_k$ are not all tangent to the same hyperplane at $Z$.  Then for some open neighborhood $\mathcal{O}'$ of $Z$ in $\mathcal{O}$, $M_k$ are smooth (real analytic) hypersurfaces-with-boundary of $\mathcal{O}'$ and $\Gamma$ is a smooth (real analytic) $(n-1)$-dimensional submanifold of $\mathcal{O}'$. 
\end{thm}

\section{Proof of Theorem \ref{mainthm}}

Let $Z \in N$.  Identify $T_Z N$ with $\mathbb{R}^{n+1}$ via a linear isometry and let $(x_1,x_2,\ldots,x_{n+1})$ denote the corresponding normal coordinates for $N$ at $Z$.  Let $\exp_Z : T_Z N \rightarrow N$ denote the exponential map of $N$ at $Z$.  Let 
\begin{equation*}
	g(x_1,x_2,\ldots,x_{n+1}) = \sum_{i,j=1}^{n+1} g_{i,j}(x_1,x_2,\ldots,x_{n+1}) dx_i \otimes dx_j
\end{equation*} 
denote the metric on $N$.  Since for every $X$ in a normal neighborhood of $Z$ in $N$, the pushforward of the exponential map $d(\exp_Z)_{\exp_Z^{-1}(X)} : \mathbb{R}^{n+1} \rightarrow T_X N$ is a linear isomorphism, we can identify $T_X N$ with $\mathbb{R}^{n+1}$ to let 
\begin{equation*}
	\|v\|_{g(x_1,x_2,\ldots,x_{n+1})} = \left( \sum_{i,j=1}^{n+1} g_{ij}(x_1,x_2,\ldots,x_{n+1}) v_i v_j \right)^{1/2}
\end{equation*} 
for every $v \in \mathbb{R}^{n+1}$. 

Let $\Omega$ be a connected open set in $\mathbb{R}^n$ such that $0 \in \Omega$ and the diameter of $\Omega$ is less than the injectivity radius of $N$ at $Z$.  Let $\gamma$ be an $(n-1)$-dimensional $C^{1,\mu}$ submanifold in $\Omega$ such that $0 \in \gamma$, $\gamma$ is tangent to $\mathbb{R}^{n-1} \times \{0\}$ at $0$, and $\Omega \setminus \gamma$ has exactly two connected components, $\Omega_+$ and $\Omega_-$.  Label $\Omega_+$ and $\Omega_-$ so that $(0,0,\ldots,0,1)$ points into $\Omega_+$ and out of $\Omega_-$ at the origin.  Let $s$ and $q$ be integers such that $q \geq 3$ and $1 \leq s < q$.  We consider the collection of hypersurfaces 
\begin{equation*}
	M_k = \exp_Z(\op{graph} u_k) 
\end{equation*}
for $u_k \in C^{1,\mu}(\Omega_+ \cup \gamma) \cap C^{\infty}(\Omega_+)$ for $k = 1,2,\ldots,s$ and 
\begin{equation*}
	M_k = \exp_Z(\op{graph} u_k) 
\end{equation*}
for $u_k \in C^{1,\mu}(\Omega_- \cup \gamma) \cap C^{\infty}(\Omega_-)$ for $k = s+1,s+2,\ldots,q$ such that $u_1,u_2,\ldots,u_q$ satisfy the following.  Assume $\max_{k=1,\ldots,s} \sup_{\Omega_+} |u_k|$ and $\max_{k=s+1,\ldots,q} \sup_{\Omega_-} |u_k|$ are small enough that each $M_k$ is properly defined and is contained in a normal neighborhood of $Z$ in $N$.  Each $u_k$ satisfies the minimal surface equation, 
\begin{align} \label{mse}
	&\sum_{i,j=1}^n D_{x_i} \left( \sqrt{\det G(x,u_k,Du_k)} \, G^{i,j}(x,u_k,Du_k) (g_{n+1,n+1}(x,u_k) D_{x_j} u_k + g_{j,n+1}(x,u_k)) \right) \\
		&- \frac{1}{2} \sum_{i,j=1}^{n+1} \sqrt{\det G(x,u_k,Du_k)} \, G^{i,j}(x,u_k,Du_k) \, D_z G_{i,j}(x,u_k,Du_k) = 0 
		\text{ in } \Omega_+ \text{ for } k = 1,2,\ldots,s, \nonumber \\
	&\sum_{i,j=1}^n D_{x_i} \left( \sqrt{\det G(x,u_k,Du_k)} \, G^{i,j}(x,u_k,Du_k) (g_{n+1,n+1}(x,u_k) D_{x_j} u_k + g_{j,n+1}(x,u_k)) \right) \nonumber \\
		&- \frac{1}{2} \sum_{i,j=1}^{n+1} \sqrt{\det G(x,u_k,Du_k)} \, G^{i,j}(x,u_k,Du_k) \, D_z G_{i,j}(x,u_k,Du_k) = 0 
		\text{ in } \Omega_- \text{ for } k = s+1,\ldots,q, \nonumber 
\end{align}
where $G(x,z,p) = (G_{i,j}(x,z,p))_{i,j=1,\ldots,n}$ is the $n \times n$ matrix given by 
\begin{equation*}
	G_{i,j}(x,z,p) = g_{i,j}(x,z) + g_{i,n+1}(x,z) p_j + g_{j,n+1}(x,z) p_i + g_{n+1,n+1}(x,z) p_i p_j 
\end{equation*}
for $(x,z) \in \mathbb{R}^n \times \mathbb{R}$ near $(0,0)$, $p \in \mathbb{R}^n$, and $i,j = 1,\ldots,n$ and $G(x,z,p)^{-1} = (G^{i,j}(x,z,p))_{i,j=1,\ldots,n}$.  Along $\gamma$, $u_1,u_2,\ldots,u_q$ satisfy 
\begin{equation} \label{coincide}
	u_1 = u_2 = \cdots = u_q \text{ on } \gamma  
\end{equation}
so that the $M_k$ have a common boundary $\Gamma = \exp_Z(\op{graph} u_1 |_{\gamma})$.  We assume that $M_1,M_2,\ldots,M_q$ satisfy (\ref{mainthm_tangents}) for some positive integers $\theta_1,\theta_2,\ldots,\theta_q$.  At the origin, we may assume that each $u_k$ satisfies 
\begin{equation} \label{gradatorigin}
	u_k(0) = 0, \hspace{10mm} D_{x_i} u_k(0) = 0 \text{ for } i = 1,\ldots,n-1,  
\end{equation}
for $k = 1,2,\ldots,q$.  If $Du_1(0) = Du_2(0) = \cdots = Du_s(0)$ and $Du_{s+1}(0) = Du_{s+2} = \cdots = Du_q(0)$, then $Du_1(0) = Du_q(0)$ by (\ref{mainthm_tangents}) and thus the hypersurfaces $M_k$ are all tangent to the same hyperplane at $Z$.  Therefore we may assume $s \geq 2$ and 
\begin{equation} \label{twoplaneassumption}
	D_{x_n} u_1(0) > D_{x_n} u_2(0) 
\end{equation}
so that the hypersurface $M_k$ are not all tangent to the same hyperplane at $Z$.

We want to express (\ref{mainthm_tangents}) in terms of differential equations.  By (\ref{gradatorigin}), we may assume 
\begin{equation} \label{invertible_chi}
	\text{the } n \times n \text{ matrix } (g_{i,j}(x,u_k(x)) + g_{j,n+1}(x,u_k(x)) D_{x_i} u_k(x))_{i,j=1,2,\ldots,n} \text{ is invertible}
\end{equation}
at each $x \in \gamma$ near the origin and $k = 1,2,\ldots,q$.  Define $\chi(x,z,p) = (\chi^1(x,z,p), \chi^2(x,z,p), \ldots, \chi^n(x,z,p),1)$ by letting $\chi^1(x,z,p)$, $\chi^2(x,z,p),\ldots,\chi^n(x,z,p)$ be the unique solutions to the linear system 
\begin{equation} \label{defn_chi}
	\sum_{j=1}^n (g_{i,j}(x,z) + g_{j,n+1}(x,z) p_i) \chi^j(x,z,p) = -g_{i,n+1}(x,z) - g_{n+1,n+1}(x,z) p_i \text{ for } i = 1,2,\ldots,n 
\end{equation}
whenever the $n \times n$ matrix $(g_{i,j}(x,z) + g_{j,n+1}(x,z) p_i)_{i,j=1,2,\ldots,n}$ is invertible so that \\ $d(\exp_Z)_{(x,u_k(x))} (\chi(x,u_k(x),Du_k(x)))$ is a normal vector to the graph of $u_k$ at $(x,u_k(x))$ for each $x \in \gamma$.  (In the special case $N = \mathbb{R}^n$, $\chi(x,u_k(x),Du_k(x)) = (-Du_k(x),1)$.)  By a rotation of $T_{\exp_Z(x,u_1(x))} N$ fixing $T_{\exp_Z(x,u_1(x))} \Gamma$ and rotating the orthogonal complement of $T_{\exp_Z(x,u_1(x))} \Gamma$ by $\pi/2$ radians, (\ref{mainthm_tangents}) is equivalent to 
\begin{equation} \label{tangents}
	\sum_{k=1}^s \theta_k \frac{\chi(x,u_k,Du_k)}{\|\chi(x,u_k,Du_k)\|_{g(x,u_k(x))}} 
	- \sum_{k=s+1}^q \theta_k \frac{\chi(x,u_k,Du_k)}{\|\chi(x,u_k,Du_k)\|_{g(x,u_k(x))}} = 0 
\end{equation}
at each $x \in \gamma$.  By (\ref{coincide}) and (\ref{twoplaneassumption}), for $x \in \gamma$ near the origin, each $\chi(x,u_k(x),Du_k(x))$ is in the span of $\chi(x,u_1(x),Du_1(x))$ and $\chi(x,u_2(x),Du_2(x))$ for $k = 1,\ldots,q$.  The span of $\chi(0,u_1(0),Du_1(0))$ and $\chi(0,u_2(0),Du_2(0))$ is $\{0\} \times \mathbb{R}^2$, so for $x \in \gamma$ near the origin the orthogonal projection of the span of $\chi(x,u_1(x),Du_1(x))$ and $\chi(x,u_2(x),Du_2(x))$ onto $\{0\} \times \mathbb{R}^2$ is bijective.  Thus by taking the $n$-th and $(n+1)$-th components of both sides of (\ref{tangents}), (\ref{tangents}) is equivalent to 
\begin{align} \label{tangents2}
	&\sum_{k=1}^s \theta_k \frac{\chi^n(x,u_k,Du_k)}{\|\chi(x,u_k,Du_k)\|_{g(x,u_k(x))}} 
		- \sum_{k=s+1}^q \theta_k \frac{\chi^n(x,u_k,Du_k)}{\|\chi(x,u_k,Du_k)\|_{g(x,u_k(x))}} = 0, \nonumber \\
	&\sum_{k=1}^s \theta_k \frac{1}{\|\chi(x,u_k,Du_k)\|_{g(x,u_k(x))}} 
		- \sum_{k=s+1}^q \theta_k \frac{1}{\|\chi(x,u_k,Du_k)\|_{g(x,u_k(x))}} = 0, 
\end{align}
at each $x \in \gamma$ near the origin.  By replacing $\Omega$ with a smaller neighborhood of the origin if necessary, assume (\ref{invertible_chi}) and (\ref{tangents2}) holds at every $x \in \gamma$.  

Now our goal is prove that for solutions $u_1,u_2,\ldots,u_q$ to the free boundary problem (\ref{mse}), (\ref{coincide}), and (\ref{tangents2}) satisfying the conditions (\ref{gradatorigin}) and (\ref{twoplaneassumption}) at the origin, $u_1,u_2,\ldots,u_q$ are smooth (real analytic) functions up to the boundary $\gamma$ and $\gamma$ is a smooth (real analytic) $(n-2)$-dimensional submanifold in $\Omega$.  We will use the partial Legendre transform of Kinderlehrer, Nirenberg, and Spuck~\cite{KNS}.  Let $w = u_1 - u_2$.  Consider the transformation $y_i = x_i$ for $i = 1,\ldots,n-1$ and $y_n = w(x)$ for $x \in \Omega_+ \cup \gamma$.  Let $U$ and $S$ denote the images of $\Omega_+$ and $\gamma$ respectively under this transformation and observe that $S \subseteq \{ y : y_n = 0\}$ by (\ref{coincide}).  By (\ref{twoplaneassumption}), $x \mapsto (x_1,\ldots,x_{n-1},w(x))$ is invertible near the origin and thus we may assume that $x \mapsto (x_1,\ldots,x_{n-1},w(x))$ is invertible on $\Omega \cup \gamma$.  The inverse transformation of $y_i = x_i$ for $i = 1,\ldots,n-1$ and $y_n = w(x)$ for $x \in \Omega_+ \cup \gamma$ is given by $x_i = y_i$ for $i = 1,\ldots,n-1$ and $x_n = \psi(y)$ for $y \in U \cup S$ for some function $\psi \in C^{1,\mu}(U \cup S) \cap C^{\infty}(U)$.  For $y \in U \cup S$, we have the tranformation $x_i = y_i$ for $i = 1,\ldots,n-1$ and $x_n = \psi(y) - Cy_n$ for some constant $C > 0$ such that $D_{y_n} \psi < C$ on $U$.  By replacing $\Omega$ with a smaller open neighborhood of the origin if necessary, we may assume that $y \mapsto (y_1,y_2,\ldots,y_{n-1},\psi(y))$ is a bijection from $U \cup S$ to $\Omega_+ \cup \gamma$ and $y \mapsto (y_1,y_2,\ldots,y_{n-1},\psi(y) - Cy_n)$ is a bijection from $U \cup S$ to $\Omega_- \cup \gamma$.  It is readily computed that 
\begin{align} \label{transform_deriv}
	&\frac{\partial}{\partial x_i} = \frac{\partial}{\partial y_i} - \frac{D_{y_i} \psi}{D_{y_n} \psi} \frac{\partial}{\partial y_n} 
		\text{ for } i = 1,\ldots,n-1, 
	&&\frac{\partial}{\partial x_n} = \frac{1}{D_{y_n} \psi} \frac{\partial}{\partial y_n}, \nonumber \\
	&\frac{\partial w}{\partial x_i} = - \frac{D_{y_i} \psi}{D_{y_n} \psi} \text{ for } i = 1,\ldots,n-1, 
	&&\frac{\partial w}{\partial x_n} = \frac{1}{D_{y_n} \psi}, 
\end{align}
for $x \in \Omega_+ \cup \gamma$ and 
\begin{align} \label{transform_deriv2}
	&\frac{\partial}{\partial x_i} = \frac{\partial}{\partial y_i} - \frac{D_{y_i} \psi}{D_{y_n} \psi - C} \frac{\partial}{\partial y_n} 
		\text{ for } i = 1,\ldots,n-1, 
	&&\frac{\partial}{\partial x_n} = \frac{1}{D_{y_n} \psi - C} \frac{\partial}{\partial y_n}, 
\end{align}
for $x \in \Omega_- \cup \gamma$.  Let 
\begin{align*}
	\phi_k(y) &= u_k(y_1,y_2,\ldots,y_{n-1},\psi(y)) \hspace{13.8mm} \text{on } U \cup S \text{ for } k = 2,3,\ldots,s, \\
	\phi_k(y) &= u_k(y_1,y_2,\ldots,y_{n-1},\psi(y)-Cy_n) \text{ on } U \cup S \text{ for } k = s+1,s+2,\ldots,q. 
\end{align*}
By (\ref{mse}), $\psi, \phi_2. \phi_3, \ldots, \phi_q$ satisfy equations of the form 
\begin{align} \label{mse_trans} 
	\sum_{i=1}^n D_{y_i} F_1^i(y, \psi, \phi_2, D\psi, D\phi_2) + F_1^0(y, \psi, \phi_2, D\psi, D\phi_2) &= 0 \text{ in } U, \nonumber \\
	\sum_{i=1}^n D_{y_i} F_k^i(y, \phi_k, D\psi, D\phi_k) + F_k^0(y, \phi_k, D\psi, D\phi_k) &= 0 \text{ in } U \text{ for } k = 2,3,\ldots,q, 
\end{align}
for some smooth (real analytic) functions $F_k^i$, $i = 0,1,2,\ldots,n$ and $k = 1,2,\ldots,q$.  By (\ref{coincide}), 
\begin{equation} \label{coincide_trans}
	\phi_2 = \phi_3 = \phi_4 = \cdots = \phi_q \text{ on } S. 
\end{equation}
By (\ref{tangents2}), $\psi, \phi_2. \phi_3, \ldots, \phi_q$ satisfy equations of the form 
\begin{align} \label{tangents_trans}
	\Phi_1(y,\psi,\phi_2,\phi_3,\ldots,\phi_q,D\psi,D\phi_2,D\phi_3,\ldots,D\phi_q) &= 0 \text{ on } S, \nonumber \\
	\Phi_2(y,\psi,\phi_2,\phi_3,\ldots,\phi_q,D\psi,D\phi_2,D\phi_3,\ldots,D\phi_q) &= 0 \text{ on } S, 
\end{align}
for some smooth (real analytic) functions $\Phi_1$ and $\Phi_2$. 

Consider the general differential system in functions $v_1,v_2\ldots,v_q$ of the form 
\begin{align} \label{system}
	\sum_{|\alpha| \leq s_k-l} D^{\alpha} F_k^{\alpha}(y,\{D^{\beta} v_j\}_{j=1,\ldots,q,|\beta| \leq t_j+l}) &= 0 
		\text{ weakly in } U \text{ for } k = 1,2,\ldots,q \text{ such that } s_k > l, \nonumber \\
	F_k(y,\{D^{\beta} v_j\}_{j=1,\ldots,q,|\beta| \leq s_k+t_j}) &= 0 
		\text{ in } U \text{ for } k = 1,2,\ldots,q \text{ such that } s_k \leq l, \nonumber \\
	\Phi_r(y,\{D^{\kappa} v_j\}_{j=1,\ldots,q,|\kappa| \leq t_j+r_h}) &= 0 \text{ on } S \text{ for } h = 1,2,\ldots,m, 
\end{align}
where $F_k^{\alpha}$, $F_k$, and $\Phi_r$ are smooth real-valued functions, $l \leq 0$ is an integer, and $s_1,\ldots,s_q$, $t_1,\ldots,t_q$, and $r_1,\ldots,r_m$ are integer weights such that $\max_k s_k = 0$, $\min_j t_j \geq -l$, $\min_{k,j} (s_k+t_j) \geq 0$, and $\min_{j,h} (t_j+r_h) \geq 0$.  The linearization of (\ref{system}) is the linear system in functions $\bar v_1,\ldots,\bar v_q$ given by  
\begin{align*}
	&\sum_{j=1}^N \sum_{|\alpha| \leq s_k-l} \sum_{|\beta| \leq t_j+l} D^{\alpha} (a_{kj}^{\alpha \beta}(y) D^{\beta} \bar v_j) 
		= \left. \frac{d}{dt} \sum_{|\alpha| \leq s_k-l} D^{\alpha} F_k^{\alpha}(y,\{D^{\beta} v_j + t D^{\beta} \bar v_j 
		\}_{j=1,\ldots,q,|\beta| \leq t_j+l}) \right|_{t=0} \\ &\hspace{62mm} = 0 \text{ in } U \text{ if } s_k > l, \\
	&\sum_{j=1}^N \sum_{|\beta| \leq s_k+t_j} a_{kj}^{\beta}(y) D^{\beta} \bar v_j 
		= \left. \frac{d}{dt} F_k(y,\{D^{\beta} v_j + t D^{\beta} \bar v_j \}_{j=1,\ldots,q,|\beta| \leq s_k+t_j-|\alpha|}) \right|_{t=0} = 0 
		\text{ in } U \text{ if } s_k \leq l, \\
	&\sum_{j=1}^N \sum_{|\kappa| \leq t_j+r_h} b_{hj}^{\kappa}(y) D^{\kappa} \bar v_j 
		= \left. \frac{d}{dt} \Phi_r(y,\{D^{\kappa} v_j + t D^{\kappa} \bar v_j\}_{j=1,\ldots,q,|\kappa| \leq t_j+r_h}) \right|_{t=0} = 0 \text{ on } S, 
\end{align*}
for $k = 1,2,\ldots,q$ and $h = 1,2,\ldots,m$, where $a_{kj}^{\alpha \beta}$ and $a_{kj}^{\beta}$ are real-valued functions on $U$ and $b_{rj}^{\kappa}$ are real-valued functions on $S$.  Let 
\begin{align*}
	L'_{kj}(y,D) &= \sum_{|\alpha| = s_k-l} \sum_{|\beta| = t_j-l} a_{kj}^{\alpha \beta}(y) D^{\alpha+\beta} 
		\text{ for } y \in U, \, k = 1,2,\ldots,q \text{ such that } s_k > l, \\
	L'_{kj}(y,D) &= \sum_{|\beta| = s_k+t_j} a_{kj}^{\beta}(y) D^{\alpha+\beta} 
		\text{ for } y \in U, \, k = 1,2,\ldots,q \text{ such that } s_k \leq l, \\
	B'_{hj}(y,D) &= \sum_{|\kappa| = t_j+r_h} b_{hj}^{\kappa}(y) D^{\kappa} \text{ for } y \in S, \, h = 1,2,\ldots,m, 
\end{align*}
for $j = 1,2,\ldots,N$ so that $L'_{kj}(y,D) \bar v_j$ and $B'_{hj}(y,D) \bar v_j$ are the principle parts of the linearization of (\ref{system}).  We say (\ref{system}) is \textit{elliptic} at $y = y_0$ if the linear system 
\begin{equation*}
	L'_{kj}(y_0,D) \bar v_j = 0 \text{ in } \mathbb{R}^n \text{ for } k = 1,2,\ldots,q 
\end{equation*}
has no nontrivial complex-valued solutions of the form $\bar v_j = c_j e^{i\xi \cdot y}$ for some $\xi \in \mathbb{R}^n \setminus \{0\}$ and $c_j \in \mathbb{C}$ for $j = 1,2,\ldots,q$.  Assuming (\ref{system}) is elliptic at the $y = y_0$, we say (\ref{system}) is \textit{coercive} or satisfies the \textit{complementing condition} at $y = y_0$ if $2m = \sum_{j=1}^q s_j + \sum_{k=1}^q t_k$ and the system 
\begin{align*}
	L'_{kj}(y_0,D) \bar v_j &= 0 \text{ in } \{y : y_n > 0\} \text{ for } k = 1,2,\ldots,q, \\
	B'_{hj}(y_0,D) \bar v_j &= 0 \text{ on } \{y : y_n = 0\} \text{ for } h = 1,2,\ldots,m, 
\end{align*}
has no nontrivial complex-valued solutions of the form $\bar v_j = c_j e^{i\xi' \cdot y' - \lambda_j y_n}$ for some $\xi' \in \mathbb{R}^{n-1}$, $c_j \in \mathbb{C}$, $\lambda_j \in \mathbb{C}$ with $\op{Re} \lambda_j > 0$ for $j = 1,\ldots,q$, where $y' = (y_1,\ldots,y_{n-1})$. 

Now consider the differential system given by (\ref{mse_trans}), (\ref{tangents_trans}), and 
\begin{equation} \label{coincide_trans2}
	\Phi_h(y,\psi,\phi_2,\phi_3,\ldots,\phi_q) \equiv \phi_h - \phi_2 = 0 \text{ on } S \text{ for } h = 3,\ldots,q
\end{equation} 
with weights $s_k = 0$ for all $k = 1,\ldots,q$, $t_j = 2$ for all $j = 1,\ldots,q$, $r_1 = r_2 = -1$, and $r_h = -2$ for $h = 3,\ldots,q$.  To complete the proof of Theorem \ref{mainthm} we must show this differential system is elliptic and coercive at the origin.  Having shown that, as was remarked in~\cite{KNS}, one can establish a Schauder estimate for linear systems of the form (\ref{system}) analogous to Lemma 9.1 of~\cite{ADK1} by a similar proof using ideas from~\cite{ADK2} and then apply the Schauder estimate in a standard difference quotient argument to show that $\psi, \phi_2, \phi_3,\ldots, \phi_q$ are $C^{2,\mu}$ functions up to the boundary on a relatively open neighborhood of the origin $U \cup S$.  Thus we can apply Theorem 6.8.2 of~\cite{Morrey} to conclude that if $(N,g)$ is a smooth (real analytic) Riemannian manifold then $\psi, \phi_2, \phi_3,\ldots, \phi_q$ are smooth (real analytic) functions up to the boundary $S$ on a relatively open neighborhood of the origin in $U \cup S$.  It follows that $u_1,\ldots,u_s$ are smooth (real analytic) up to the boundary $\gamma$ on a relatively open neighborhood of the origin in $\Omega_+ \cup \gamma$, $u_{s+1},u_{s+2},\ldots,u_q$ are smooth (real analytic) up to the boundary $\gamma$ on a relatively open neighborhood of the origin in $\Omega_- \cup \gamma$, and $\Gamma = \{(y',\psi(y',0),\phi_2(y',0)) : y' = (y_1,\ldots,y_{n-1}) \in S \}$ is a smooth (real analytic) $(n-1)$-dimensional submanifold near the origin. 

Let $a_k = D_{x_n} u_k(0)$ for $k = 1,2,\ldots,q$.  By (\ref{gradatorigin}), (\ref{transform_deriv}), and (\ref{transform_deriv2}), 
\begin{gather} \label{gradatorigin_trans}
	D_{y_i} \psi(0) = D_{y_i} \phi_2(0) = D_{y_i} \phi_3(0) = \cdots = D_{y_i} \phi_q(0) = 0 \text{ for } i = 1,2,\ldots,n-1, \\
	a_1 - a_2 = \frac{1}{D_{y_n} \psi(0)}, \hspace{4mm} 
	a_k = \frac{D_{y_n} \phi_k(0)}{D_{y_n} \psi(0)} \text{ if } k = 2,3,\ldots,s, \hspace{4mm} 
	a_k = \frac{D_{y_n} \phi_k(0)}{D_{y_n} \psi(0) - C} \text{ if } k = s+1,\ldots,q.  \nonumber 
\end{gather}
We want to linearize and take the principle part of (\ref{mse_trans}) at the origin.  Consider the equation for $k = 2$ in (\ref{mse_trans}).  We can rewrite the minimal surface equation for $u_2$ from (\ref{mse}) as 
\begin{align*}
	&\sum_{i,j=1}^n \sqrt{\det G(x,u_2,Du_2)} \, G^{i,j}(x,u_2,Du_2) g_{n+1,n+1}(x,u_2) D_{x_i x_j} u_2 
	\\&+ \frac{1}{2} \sum_{i,j,k,l=1}^n \sqrt{\det G(x,u_2,Du_2)} (G^{i,j}(x,u_2,Du_2) G^{k,l}(x,u_2,Du_2) - 2G^{i,k}(x,u_2,Du_2) G^{l,j}(x,u_2,Du_2)) 
		\\&\cdot (g_{k,n+1}(x,u_2) D_{x_i x_l} u_2 + g_{l,n+1}(x,u_2) D_{x_i x_k} u_2 + g_{n+1,n+1}(x,u_2) D_{x_k} u_2 D_{x_i x_l} u_2 
		\\&+ g_{n+1,n+1}(x,u_2) D_{x_l} u_2 D_{x_i x_k} u_2) (g_{n+1,n+1}(x,u_2) D_{x_j} u_2 + g_{j,n+1}(x,u_2)) + R(x,u_2,Du_2) = 0 \text{ in } \Omega_+
\end{align*}
for some function $R(x,z,p)$, using the fact that $u_2 \in C^{\infty}(\Omega_+)$.  By (\ref{gradatorigin}), $g_{ii}(0,u_2(0)) = 1$ for $i = 1,2,\ldots,n+1$, $g_{ij}(0,u_2(0)) = 0$ if $i \neq j$, $G_{i,i}(0,u_2(0),Du_2(0)) = 1$ for $i = 1,2,\ldots,n-1$, $G_{n,n}(0,u_2(0),Du_2(0)) = 1+a_2^2$, and $G_{i,j}(0,u_2(0),Du_2(0)) = 0$ for $i \neq j$.  Thus linearizing and taking the principle part of the equation for $k = 2$ in (\ref{mse_trans}) yields 
\begin{equation} \label{mse_lin0}
	(1+a_2^2) \sum_{i=1}^{n-1} \overline{D_{x_i x_i} u_2} + \overline{D_{x_n x_n} u_2} = 0 \text{ on } \{ y : y_n > 0 \}, 
\end{equation}
where for $i = 1,2,\ldots,n$ we let $\overline{D_{x_i x_i} u_2}$ denote the result of rewriting $D_{x_i x_i} u_2$ as a function of $y$ and then computing its linearization and second order principle part at the origin.  By (\ref{transform_deriv}) and (\ref{gradatorigin_trans}), 
\begin{align*}
	\overline{D_{x_i x_i} u_2} &= D_{y_i y_i} \bar \phi_2 - \frac{D_{y_n} \phi_2(0)}{D_{y_n} \psi(0)} D_{y_n y_n} \bar \psi 
		= D_{y_i y_i} (\bar \phi_2 - a_2 \bar \psi) \\
	\overline{D_{x_n x_n} u_2} &= \frac{1}{D_{y_n} \psi(0)^2} D_{y_n y_n} \bar \phi_2 - \frac{D_{y_n} \phi_2(0)}{D_{y_n} \psi(0)^3} D_{y_n y_n} \bar \psi 
		= (a_1-a_2)^2 D_{y_n y_n} (\bar \phi_2 - a_2 \bar \psi)
\end{align*}
for functions $\bar \psi$ and $\bar \phi_2$, so we can write (\ref{mse_lin0}) as 
\begin{equation*}
	(1+a_2^2) \sum_{i=1}^{n-1} D_{y_i y_i} (\bar \phi_2 - a_2 \bar \psi) + (a_1 - a_2)^2 D_{y_n y_n} (\bar \phi_2 - a_2 \bar \psi) = 0 
	\text{ in } \{y : y_n > 0\}. 
\end{equation*}
By similar computations, we can linearize and take the principle part of the equations in (\ref{mse_trans}) for every $k \in \{1,2,\ldots,q\}$ using (\ref{gradatorigin}), (\ref{transform_deriv}), (\ref{transform_deriv2}), and (\ref{gradatorigin_trans}) to obtain the differential system in $\bar \psi, \bar \phi_2, \bar \phi_3, \ldots, \bar \phi_q$ of 
\begin{align} \label{mse_lin}
	(1+a_1^2) \sum_{i=1}^{n-1} D_{y_i y_i} (\bar \phi_2 - a_1 \bar \psi) + (a_1 - a_2)^2 D_{y_n y_n} (\bar \phi_2 - a_1 \bar \psi) &= 0, \\
	(1+a_k^2) \sum_{i=1}^{n-1} D_{y_i y_i} (\bar \phi_k - a_k \bar \psi) + (a_1 - a_2)^2 D_{y_n y_n} (\bar \phi_k - a_k \bar \psi) &= 0 
		\text{ for } k = 2,3,\ldots,s, \nonumber \\
	(1+a_k^2) \sum_{i=1}^{n-1} D_{y_i y_i} (\bar \phi_k - a_k \bar \psi) + \left( \frac{a_1 - a_2}{1 - C(a_1 - a_2)} \right)^2 
		D_{y_n y_n} (\bar \phi_k - a_k \bar \psi) &= 0 \text{ for } k = s+1,\ldots,q, \nonumber
\end{align}
in $\{y : y_n > 0\}$, which is obviously an elliptic system in $\bar \phi_2 - a_1 \bar \psi$ and $\bar \phi_k - a_k \bar \psi$ for $k = 2,3,\ldots,q$.  To check coercivity, it suffices to consider solutions to (\ref{mse_lin}) of the form 
\begin{equation} \label{barpsi_eqn}
	\bar \phi_2 - a_1 \bar \psi = c_1 e^{i\xi' \cdot y' - \lambda_1 y_n}, \hspace{10mm} 
	\bar \phi_k - a_k \bar \psi = c_k e^{i\xi' \cdot y' - \lambda_k y_n} \text{ for } k = 2,3,\ldots,q, 
\end{equation}
where $\xi' \in \mathbb{R}^{n-1}$, $c_k \in \mathbb{C}$, and $\lambda_k \in \mathbb{C}$ with $\op{Re} \lambda_k > 0$ for $k = 1,2,\ldots,q$.  By (\ref{mse_lin}), 
\begin{align} \label{nu_eqn}
	\lambda_k &= \frac{\sqrt{1+a_k^2} |\xi'|}{a_1-a_2} \text{ for } k = 1,2,\ldots,s, \nonumber \\
	\lambda_k &= \frac{(C(a_1 - a_2)-1) \sqrt{1+a_k^2} |\xi'|}{a_1-a_2} \text{ for } k = s+1,\ldots,q. 
\end{align}
Since we assume $\lambda_k > 0$ for all $k = 1,2,\ldots,q$, $\xi' \neq 0$.  The linearization of (\ref{coincide_trans}) simply yields
\begin{equation} \label{coincide_lin} 
	\bar \phi_2 = \bar \phi_3 = \bar \phi_4 = \cdots = \bar \phi_q \text{ on } \{ y : y_n = 0 \}. 
\end{equation}
By substituting (\ref{barpsi_eqn}) in (\ref{coincide_lin}) and solving for $c_k$ in terms of $c_1$ and $c_2$ for $k = 3,4,\ldots,q$, we obtain 
\begin{equation} \label{coincide_lin2}
	c_k = \frac{a_k - a_2}{a_1 - a_2} c_1 + \frac{a_1 - a_k}{a_1 - a_2} c_2 \text{ for } k = 1,2,\ldots,q, 
\end{equation}
which we note also holds for $k = 1,2$ trivially.  Next we want to linearize and take the principle part of (\ref{tangents_trans}) at the origin.  By (\ref{defn_chi}) and the fact that $g_{ii}(0,u_k(0)) = 1$ for $i = 1,2,\ldots,q$ and $g_{ij}(0,u_k(0)) = 0$ for $i \neq j$ for all $k = 1,2,\ldots,q$, $\chi(0,u_k(0),Du_k(0)) = (0,0,\ldots,0,-a_k,1)$ for $k = 1,2,\ldots,q$ and the result of rewriting $\chi(x,u_k,Du_k)$ as a function of $y$ and computing the first order principle part of its linearization at the origin is $(-\overline{D_{x_1} u_k}, -\overline{D_{x_2} u_k},\ldots, -\overline{D_{x_n} u_k}, 0)$, where for $k = 1,2,\ldots,q$ and $i = 1,2,\ldots,n$ we let $\overline{D_{x_i} u_k}$ denote the result of rewriting $D_{x_i} u_k$ as a function of $y$ and then computing the first order principle part of its linearization at the origin.  Hence the linearizing and taking the principle part of (\ref{tangents2}) is 
\begin{align} \label{tangents_lin0} 
	&\sum_{k=1}^s \theta_k \frac{a_k \overline{D_{x_n} u_k}}{(1+a_k^2)^{3/2}} - \sum_{k=s+1}^q \theta_k \frac{a_k \overline{D_{x_n} u_k}}{(1+a_k^2)^{3/2}} = 0, 
		\nonumber \\
	&\sum_{k=1}^s \theta_k \frac{\overline{D_{x_n} u_k}}{(1+a_k^2)^{3/2}} - \sum_{k=s+1}^q \theta_k \frac{\overline{D_{x_n} u_k}}{(1+a_k^2)^{3/2}} = 0, 
\end{align}
on $\{ y : y_n = 0 \}$.  By (\ref{transform_deriv}), (\ref{transform_deriv2}), and (\ref{gradatorigin_trans}), 
\begin{align} \label{tangents_lin0b} 
	\overline{D_{x_n} u_1} &= \frac{1}{D_{y_n} \psi(0)} D_{y_n} \bar \phi_2 - \frac{1+D_{y_n} \phi_2(0)}{D_{y_n} \psi(0)^2} D_{y_n} \bar \psi 
		= (a_1-a_2) D_{y_n} (\bar \phi_2 - a_1 \bar \psi), \nonumber \\
	\overline{D_{x_n} u_k} &= \frac{1}{D_{y_n} \psi(0)} D_{y_n} \bar \phi_k - \frac{D_{y_n} \phi_k(0)}{D_{y_n} \psi(0)^2} D_{y_n} \bar \psi 
		= (a_1-a_2) D_{y_n} (\bar \phi_k - a_k \bar \psi) \text{ for } k = 2,3,\ldots,s, \nonumber \\
	\overline{D_{x_n} u_k} &= \frac{1}{D_{y_n} \psi(0)-C} D_{y_n} \bar \phi_k - \frac{D_{y_n} \phi_k(0)}{(D_{y_n} \psi(0)-C)^2} D_{y_n} \bar \psi 
		 \nonumber \\&= \frac{a_1-a_2}{1-C(a_1-a_2)} D_{y_n} (\bar \phi_k - a_k \bar \psi) \text{ for } k = s+1,s+2,\ldots,q.
\end{align}
Combining (\ref{tangents_lin0}) and (\ref{tangents_lin0b}) yields 
\begin{align} \label{tangents_lin}
	\theta_1 \frac{a_1 D_{y_n} (\bar \phi_2 - a_1 \bar \psi)}{(1+a_1^2)^{3/2}} 
		+ \sum_{k=2}^s \theta_k \frac{a_k D_{y_n} (\bar \phi_k - a_k \bar \psi)}{(1+a_k^2)^{3/2}} 
		- \sum_{k=s+1}^q \theta_k \frac{a_k D_{y_n} (\bar \phi_k - a_k \bar \psi)}{(1 - C(a_1-a_2)) (1+a_k^2)^{3/2}} &= 0, \nonumber \\
	\theta_1 \frac{D_{y_n} (\bar \phi_2 - a_1 \bar \psi)}{(1+a_1^2)^{3/2}} 
		+ \sum_{k=2}^s \theta_k \frac{D_{y_n} (\bar \phi_k - a_k \bar \psi)}{(1+a_k^2)^{3/2}} 
		- \sum_{k=s+1}^q \theta_k \frac{D_{y_n} (\bar \phi_k - a_k \bar \psi)}{(1 - C(a_1-a_2)) (1+a_k^2)^{3/2}} &= 0, 
\end{align}
on $\{ y : y_n = 0 \}$.  Using the fact that $D_{y_n} (\bar \phi_2 - a_1 \bar \psi) = -c_1 \lambda_1$ and $D_{y_n} (\bar \phi_k - a_k \bar \psi) = -c_k \lambda_k$ for $k = 2,3,\ldots,q$ by (\ref{barpsi_eqn}), where $\lambda_k$ are given by (\ref{nu_eqn}), we can rewrite (\ref{tangents_lin}) as 
\begin{equation*}
	\sum_{k=1}^q \frac{\theta_k a_k c_k}{1+a_k^2} = 0, \hspace{10mm} 
	\sum_{k=1}^q \frac{\theta_k c_k}{1+a_k^2} = 0. 
\end{equation*}
By (\ref{coincide_lin2}), 
\begin{align} \label{tangents_lin3}
	&\sum_{k=1}^q \frac{\theta_k a_k}{1+a_k^2} ((a_k - a_2) c_1 + (a_1 - a_k) c_2) = 0, \nonumber \\
	&\sum_{k=1}^q \frac{\theta_k}{1+a_k^2} ((a_k - a_2) c_1 + (a_1 - a_k) c_2) = 0. 
\end{align}
(\ref{tangents_lin3}) is a linear system of two equations with $c_1$ and $c_2$ as unknowns whose determinant is 
\begin{equation*}
	D = (a_1 - a_2) \left( \left( \sum_{k=1}^q \frac{\theta_k}{1+a_k^2} \right) \left( \sum_{k=1}^q \frac{\theta_k a_k^2}{1+a_k^2} \right) 
		- \left( \sum_{k=1}^q \frac{\theta_k a_k}{1+a_k^2} \right)^2 \right) .
\end{equation*}
By Cauchy-Schwartz $D \geq 0$ with $D = 0$ only if $a_1 = a_2 = \cdots = a_q$.  Since $a_1 = a_2 = \cdots = a_q$ contradicts (\ref{twoplaneassumption}), $D > 0$.  Hence (\ref{tangents_lin3}) implies that $c_1 = c_2 = 0$ and thus by (\ref{coincide_lin2}) $c_k = 0$ for all $k = 1,2,\ldots,q$.  Therefore the system (\ref{mse_lin}), (\ref{coincide_lin}), and (\ref{tangents_lin}) is coercive in $\bar \phi_2 - a_1 \bar \psi$ and $\bar \phi_k - a_k \bar \psi$ for $k = 2,3,\ldots,q$.  Consequently the differential system given by (\ref{mse_trans}), (\ref{tangents_trans}), and (\ref{coincide_trans2}) is elliptic and coercive at the origin and Theorem \ref{mainthm} follows.

\section{Appendix: Proof of Remark \ref{remark2}} \label{sec:appendix}

\begin{lemma} \label{C1toC1mu}
Let $(N,g)$ be an $(n+1)$-dimensional, smooth (real analytic), Riemannian manifold and $\mu \in (0,1/2)$.  Let $Z \in N$, $\mathcal{O}$ be an open neighborhood of $Z$ in $N$, and $\mu \in (0,1)$.  Let $q \geq 3$ and consider the stationary $n$-varifold $V$ of the form (\ref{V_form}) for distinct $C^1$ embedded hypersurfaces-with-boundary $M_1,M_2,\ldots,M_q$ of $\mathcal{O}$ that have a common boundary $\Gamma$ with $Z \in \Gamma$ and for positive integers $\theta_1,\theta_2,\ldots,\theta_q$.  Assume the hypersurfaces $M_k$ are not all tangent to the same hyperplane at $Z$.  Then for some open neighborhood $\mathcal{O}'$ of $Z$ in $\mathcal{O}$, $M_k$ are $C^{1,\mu}$ hypersurfaces-with-boundary of $\mathcal{O}'$ and $\Gamma$ is a $C^{1,\mu}$ $(n-1)$-dimensional submanifold of $\mathcal{O}'$. 
\end{lemma}
\begin{proof}
Lemma \ref{C1toC1mu} follows from the proof of the Minimum Distance Theorem of~\cite{Wic}.  The tangent varifold to $V$ at $Z$ is 
\begin{equation*}
	C_0 = \sum_{k=1}^q \theta_k |H_k^{(0)}|,
\end{equation*}
where $H_k^{(0)}$ are the tangent half-hyperplanes to $M_k$ at $Z$.  Identify $T_Z N$ with $\mathbb{R}^{n+1}$ via a linear isometry and assume $T_Z \Gamma = \{0\} \times \mathbb{R}^{n-1}$.  Let $\exp_Z : T_Z N \rightarrow N$ denote the exponential map of $N$ at $Z$ and $\mathcal{N}_{\sigma}(Z)$ denote the normal neighborhood of $Z$ in $N$ of radius $\sigma > 0$.  Let $\eta_{\sigma}(X) = X/\sigma$ for $X \in \mathbb{R}^{n+1}$ and $\sigma > 0$.  Choose $\sigma > 0$ such that $\sigma$ is less than the injectivity radius of $N$ at $Z$, $\mathcal{N}_{\sigma}(Z) \subset \mathcal{O}$, $\widehat{M_k} = \eta_{\sigma}(\exp_Z^{-1}(M_k \cap \mathcal{N}_{\sigma}(Z)))$ is $C^1$ close to $H_k^{(0)} \cap B^{n+1}_1(0)$ for $k = 1,2,\ldots,q$, and $\widehat{\Gamma} = \eta_{\sigma}(\exp_Z^{-1}(\Gamma \cap \mathcal{N}_{\sigma}(Z)))$ is $C^1$ close to $\{0\} \times B^{n-1}_1(0)$.  Let $\widehat{V} = \eta_{\sigma \#} \exp_{Z \#}^{-1} (V \llcorner \mathcal{N}_{\sigma}(Z))$.  For a small enough choice of $\sigma$, all the hypotheses of the Minimum Distance Theorem~\cite[Theorem 18.2]{Wic} hold true with $\widehat{V}$ in place of $V$, $C_0$, and $\alpha = \mu$ except $\widehat{V}$ might not belong to $\mathcal{S}_{\alpha}^{\star}$.  In particular, the $\alpha$-structure condition $(\mathcal{S}3)$ fails.  It is unnecessary to assume the interior of each $M_i$ is stable since assuming $\widehat{\Gamma}$ is close to $\{0\} \times B^{n-1}_1(0)$ implies Lemma 16.5(a) of~\cite{Wic} is true, i.e. that there is a high concentration of points of density $\geq q$ near $\{0\} \times B^{n-1}_{1/2}(0)$.  We need to modify the proof of Theorem 16.2 of~\cite{Wic}, which establishes a priori $L^2$ estimates as in~\cite{SimonCTC}, since the proof assumes $\widehat{V} \in \mathcal{S}_{\alpha}^{\star}$.  Let 
\begin{equation*}
	C = \sum_{k=1}^q \sum_{l=1}^{\theta_k} |H_{k,l}|
\end{equation*}
for some half-hyperplaces $H_{k,l}$ that are close to $H_k$ and have boundary $\{0\} \times \mathbb{R}^{n-1}$ and let 
\begin{equation*}
	T_{\kappa,\rho}(\zeta) = \{ (x,y) \in \mathbb{R}^2 \times \mathbb{R}^{n-1} : ||x| - \rho|^2 + |y-\zeta|^2 < \kappa^2 \rho^2/64 \}
\end{equation*}
for $\kappa \in (0,1]$, $\rho \in (0,1/16)$, and $\zeta \in \mathbb{R}^{n-1}$.  The proof of Theorem 16.2 claims that there is a $\delta \in (0,1/2)$ (depending on $C^{(0)}$ and independent of $\widehat{V}$ and $C$) such that if $\rho \in (0,1/16)$ and $y \in \mathbb{R}^{n-1}$ with $\rho^2 + |\zeta|^2 \leq (13/16)^2$, $\op{spt}\|\widehat{V}\| \cap T_{\rho,1/16}(\zeta) \neq \emptyset$, and 
\begin{equation} \label{simon_eqn1}
	\rho^{-n-2} \int_{T_{\rho,1}(\zeta)} \op{dist}((x,y),\op{spt}\|C\|)^2 d\|\widehat{V}\|(x,y) < \delta^2
\end{equation}
then $\widehat{V} \llcorner T_{\rho,1/2}(\zeta)$ can be written as a graph over $\op{spt}\|C\| \cap T_{\rho,1/2}(\zeta)$ of functions with small gradient.  Observe that for our particular choice of $\widehat{V}$, (\ref{simon_eqn1}) implies $\widehat{\Gamma} \cap T_{\rho,7/8}(\zeta) = \emptyset$ since $\widehat{M}_k$ is $C^1$ close to $H_k^{(0)} \cap B^{n+1}_1(0)$ and thus the claim follows by Allard regularity.  The rest of the proof of the Minimum Distance Theorem goes through without change to conclude that each $M_k$ is a $C^{1,\mu}$ hypersurface-with-boundary in some open neighborhood $\mathcal{O}'$ of $Z$.
\end{proof}

\end{document}